\documentclass[11pt, reqno]{amsart}
\usepackage{indentfirst, amssymb, amsmath, amsthm, mathrsfs, setspace, indentfirst, enumerate,  mathrsfs, amsmath, amsthm}
\usepackage[bookmarksnumbered, colorlinks, plainpages]{hyperref}
\usepackage{mathrsfs}
\newtheorem*{theoA}{Theorem A}
\newtheorem*{theoB}{Theorem B}
\newtheorem*{theoC}{Theorem C}
\newtheorem*{theoD}{Theorem D}

\newtheorem*{cor A}{Corollary A}
\newtheorem*{cor B}{Corollary B}
\newtheorem{theo}{Theorem}[section]
\newtheorem{lem}{Lemma}[section]
\newtheorem{cor}{Corollary}[section]

\newtheorem{exm}{Example}[section]

\newtheorem{rem}{Remark}[section]

\newcommand{\ol}{\overline}
\newcommand{\be}{\begin{equation}}
\newcommand{\ee}{\end{equation}}
\newcommand{\beas}{\begin{eqnarray*}}
\newcommand{\eeas}{\end{eqnarray*}}
\newcommand{\bea}{\begin{eqnarray}}
\newcommand{\eea}{\end{eqnarray}}

\numberwithin{equation}{section}
\begin{document}
\title[U\MakeLowercase{niqueness of Entire Functions Concerning}....]{\LARGE U\Large\MakeLowercase{niqueness of Entire Functions Concerning their linear differential polynomials in shift}}
\date{}
\author[J. S\MakeLowercase{arkar} \MakeLowercase{and} D. P\MakeLowercase{ramanik}]{J\MakeLowercase{eet} S\MakeLowercase{arkar}$^*$ \MakeLowercase{and} D\MakeLowercase{ebabrata} P\MakeLowercase{ramanik}}

\address{Department of Mathematics, St. Thomas College of Engineering and Technology, Kolkata, West Bengal-70003, India.}
\email{jeetsarkar.math@gmail.com}

\address{Department of Mathematics, Raiganj University, Raiganj, West Bengal-733134, India.}
\email{debumath07@gmail.com}

\renewcommand{\thefootnote}{}
\footnote{*\emph{Corresponding Author}: Jeet Sarkar.}
\setcounter{footnote}{0}
\renewcommand{\thefootnote}{\arabic{footnote}}

\begin{abstract} In the paper, we investigate the uniqueness problem of entire functions concerning their linear differential polynomial in shift and obtain three results which improve and generalize the recent result due to Qi (Ann. Polon. Math., 102 (2011), 129-142.) in a large extend.
\end{abstract}

\thanks{Typeset by \AmS -\LaTeX}
\maketitle

\section{{\bf Introduction and main results}}
We assume that the reader is familiar with standard notation and main results of Nevanlinna Theory (see \cite{WKH, YY1}). By $S(r, f)$ we denote any quantity that satisfies the condition $S(r, f) = o(T(r, f))$ as $r\to \infty$ possibly outside of an exceptional set of finite linear measure. A meromorphic function $a$ is said to be a small function of $f$ if $T(r,a)=S(r,f)$ and we denote by $S(f)$ the set of functions which are small compared to $f$.
Moreover, we use notations $\rho(f)$ and $\rho_1(f)$ for the order and the hyper-order of a meromorphic function $f$ respectively.
As usual, the abbreviation CM means ``counting multiplicities'', while IM means ``ignoring multiplicities''.

Let $k, m, n\in\mathbb{N}$ and $a_1, a_2\in S(f)$. Denote by $S_{(m,n)}(a_1)$ the set of those points $z\in\mathbb{C}$ such that $z$ is an $a_1$-point of $f$ of order $m$ and an $a_1$-point of $g$ of order $n$. The set $S_{(m,n)}(a_2)$ can be defined similarly. Let $\ol N_{(m,n)}(r,a_i;f)$ denote the reduced counting function of $f$ with respect to the set $S_{(m,n)}(a_i)$ for $i=1,2$. 

\smallskip
The research on the uniqueness problem of meromorphic function sharing values or small functions with its derivatives is an active field and the study is based on the Nevanlinna value distribution theory. The research on this topic was started by Rubel and Yang \cite{RY}. Now we state their result.

\begin{theoA}\cite{RY} Let $f$ be a non-constant entire function and $a, b\in\mathbb{C}$ such that $b\not=a$. If $f$ and $f'$ share $a$ and $b$ CM, then $f\equiv f'$.
\end{theoA} 

Mues and Steinmetz \cite{MS} further generalized \textrm{Theorem A} from sharing values CM to IM and obtained the following result.

\begin{theoB}\cite{MS} Let $f$ be a non-constant entire function and $a, b\in\mathbb{C}$ such that $b\not=a$. If $f$ and $f'$ share $a$ and $b$ IM, then $f\equiv f'$.
\end{theoB}

Recently the uniqueness of meromorphic functions sharing values with their shifts or difference operators has become a subject of great interest. Now, we recall the following result due to Heittokangas et al. \cite{HKLR1}, which is shift analogue of Theorem A.

\begin{theoC}\cite{HKLR1} Let $f$ be a non-constant entire function of finite order, $c\in\mathbb{C}\backslash \{0\}$ and $a, b\in\mathbb{C}$ be distinct. If $f(z)$ and $f(z+c)$ share $a$ and $b$ CM, then $f(z)\equiv f(z+c)$.
\end{theoC}

In 2011, Qi \cite{Qi} further improved \textrm{Theorem C} and obtained the following result.

\begin{theoD}\cite{Qi} Let $f$ be a non-constant entire function of finite order, $c\in\mathbb{C}\backslash \{0\}$ and $a, b\in\mathbb{C}$ be distinct. If $f(z)$ and $f(z+c)$ share $a$ and $b$ IM, then $f(z)\equiv f(z+c)$.
\end{theoD}

\smallskip
The the time-delay differential equation 
\[f'(x) = f(x-k),\]
$k>0$ is well known and extensively studied in real analysis, which have numerous applications ranging from cell growth models to current collection systems for an electric locomotive to wavelets. For a complex variable counterpart, Liu and Dong \cite{11} studied the complex differential-difference equation $f'(z)=f(z+c)$, where $c\in\mathbb{C}\backslash \{0\}$.
Recently, many authors have started to consider the sharing values problems of meromorphic
functions with their difference operators or shifts. Some results were considered in \cite{BM1}, \cite{BM2}, \cite{BM3}, \cite{CC2}, \cite{DM1}, \cite{FLSY}, \cite{HKT}, \cite{HKLRZ}-\cite{11}, \cite{M1}, \cite{M2}, \cite{MD}, \cite{MD1}, \cite{MP1}, \cite{MP2}, \cite{MS1}, \cite{MS2}, \cite{MJS}, \cite{MSP1}, \cite{MNS2}, \cite{Qi}-\cite{QY1}, \cite{HYX}.

We now introduce a linear differential polynomial as follows
\[\mathscr{L}_k(f)=b_0 f+b_1 f'+\ldots+b_k f^{(k)},\]
where $c\in\mathbb{C}$ and $b_j\in S(f)$ such that $(b_0, b_1,\ldots, b_k)\not\equiv (0,0,\ldots,0)$.\par

In the paper, we have extended and improved Theorem D in the following three directions:
\begin{enumerate}
\item[(1)] We replace $f(z+c)$ by $\mathscr{L}_k(f(z+c))$ in Theorem D.
\item[(2)] We consider $a$ and $b$ as the small functions of $f$ in Theorem D.
\item[(3)] We replace $\rho(f)<+\infty$ by $\rho_1(f)<1$ in Theorem D.
\end{enumerate}

Now we state our results.

\begin{theo}\label{t1} Let $f$ be a non-constant entire function with $\rho_1(f)<1$, $c\in\mathbb{C}\backslash \{0\}$ and $a_1, a_2\in S(f)$ be distinct such that either $a_1$ and $a_2$ are $c$-periodic or $a_1'-a'_2$ is non-constant. Let 
\bea\label{2} \phi=\frac{\left((a_1'-a_2')(f-a_1)-(a_1-a_2)(f'-a_1')\right)\left(f-\mathscr{L}_k(f(z+c))\right)}{(f-a_1)(f-a_2)}.\eea
If $f(z)$ and $\mathscr{L}_k(f(z+c))$ share $a_1, a_2$ IM, then either $f(z)\equiv \mathscr{L}_k(f(z+c))$ or $\sum_{i=1}^k b_i\psi_i=-b_0$,
where $\psi_1(z)=(a_1'(z+c)-a_2'(z+c)-\phi(z+c))/(a_1(z+c)-a_2(z+c))$ and $\psi_{i+1}=\psi_i'+\psi_1 \psi_i,\;i=1,2,\ldots,k-1$.
\end{theo}

When $a_1$, $a_2$ and $b_0, b_1, \ldots, b_k$ are constant, we have the following.

\begin{cor}\label{c1} Let $f$ be a non-constant entire function with $\rho_1(f)<1$, $c\in\mathbb{C}\backslash \{0\}$ and $a_1, a_2\in\mathbb{C}$ be distinct. Let $b_0, b_1, \ldots, b_k\in\mathbb{C}$ such that $(b_0, b_1, \ldots, b_k)\neq (0,0,\ldots, 0)$ and $\phi$ be defined as in (\ref{2}).
If $f(z)$ and $\mathscr{L}_k(f(z+c))$ share $a_1, a_2$ IM, then either $f(z)\equiv \mathscr{L}_k(f(z+c))$ or $\sum_{i=1}^k b_i\psi_i=-b_0$,
where $\psi_1(z)=-\phi(z+c)/(a_1-a_2)$ is a constant and $\psi_{i+1}=\psi_i'+\psi_1 \psi_i,\;i=1,2,\ldots,k-1$.
\end{cor}

When $(b_0, b_1, b_2,\ldots, b_k)=(1,0,0,\ldots,0)$, we have the following.

\begin{cor}\label{c2} Let $f$ be a non-constant entire function with $\rho_1(f)<1$, $c\in\mathbb{C}\backslash \{0\}$ and $a_1, a_2\in S(f)$ be distinct such that either $a_1$ and $a_2$ are $c$-periodic or $a_1'-a'_2$ is non-constant. If $f(z)$ and $f(z+c)$ share $a_1, a_2$ IM, then $f(z)\equiv f(z+c)$.
\end{cor}

When $(b_0, b_1, b_2,\ldots, b_k)=(0,0,0,\ldots,1)$, we have the following.
\begin{cor}\label{c3} Let $f$ be a non-constant entire function with $\rho_1(f)<1$, $c\in\mathbb{C}\backslash \{0\}$ and $a_1, a_2\in S(f)$ be distinct such that either $a_1$ and $a_2$ are $c$-periodic or $a_1'-a'_2$ is non-constant. If $f(z)$ and $f^{(k)}(z+c)$ share $a_1, a_2$ IM, then $f(z)\equiv f^{(k)}(z+c)$.
\end{cor}

Following examples show that the condition ``$\rho _1(f)< 1$'' in Theorem \ref{t1} is necessary.

\begin{exm} Let $f(z) = e^{\sin z}$, $c= \pi$, $a_1(z)=1$, $a_2(z)=-1$, $b_0(z)=\frac{1}{4}$ and $b_1(z)=-\frac{3}{4\cos z}$. Clearly $\rho _1(f)=1$ and $\mathscr{L}_1(f(z+c))= b_0(z)f(z+c)+b_1(z)f'(z+c)= e^{-\sin z}$. Now from (\ref{2}), we see that $\phi(z)=-2\cos z$  and so $\psi_1(z)=(a'_1(z+c)-a'_2(z+c)-\phi(z+c))/(a_1(z+c)-a_2(z+c))=\cos z$, but $b_1(z)\psi_1(z)=-\frac{3}{4}\neq -b_0(z)$. Clearly $f(z)$ and $\mathscr{L}_1(f(z+c))$ share $1$ and $-1$ CM, but $f$ does not satisfy any case of Theorem \ref{t1}.
\end{exm}

\begin{exm} Let $f(z)=e^{e^z}$, $c=\pi i$, $a_1=1$ and $a_2(z)=-e^z$. Note that $\rho_1(f)=1$. Clearly $f(z)$ and $f'(z+c)$ share $a_1, a_2$ IM, but $f(z)\not\equiv f'(z+c)$.
\end{exm}

Following example asserts that condition ``$a_1, a_2(\not\equiv \infty)$'' is sharp in Theorem \ref{t1}.

\begin{exm} Let $f(z) = \sin z$, $c= \pi$, $a_1(z)=0$ and $a_2(z)=\infty$. Clearly $\rho_1(f)=0<1$ and $f(z)$ and $f(z+c)$ share $a_1$ and $a_2$, but $f(z)\not\equiv f(z+c)$.
\end{exm}

Following example shows that Theorem \ref{t1} does not hold for non-constant meromorphic function $f$ with $N(r, f)\neq S(r, f)$.

\begin{exm} Let $f(z)= \frac{e^z+1}{e^{z}-1}$, $e^c=-1$, $a_1= 1$ and $a_2= -1$. Note that $\rho_1(f)=0<1$ and $N(r,f)\neq S(r,f)$. Clearly $f(z)$ and $f(z+c)$ share $a_1$ and $a_2$ CM, but $f(z)\not\equiv f(z+c)$. 
\end{exm}

\section {{\bf Auxiliary lemmas}}
In the proof of Lemma \ref{l4} below, we make use of three key lemmas. For the convenience of the reader, we recall these lemmas here.

\begin{lem}\label{l1}\cite{LY} If $f$ and $g$ are non-constant meromorphic functions, then 
\[N\big(r,f/g\big)-N\big(r,g/f\big)=N(r,f)+N(r,0;g)-N(r,g)-N(r,0;f).\]
\end{lem}

\begin{lem}\label{l2}\cite{HKT} Let $f$ be a non-constant meromorphic function with $\rho_1(f)<1$ and $c\in\mathbb{C}\backslash \{0\}$. Then
\[m\left(r,\frac{f(z+c)}{f(z)}\right)+m\left(r,\frac{f(z)}{f(z+c)}\right)=S(r,f)\]
for all $r$ outside of a possible exceptional set $E$ with finite logarithmic measure.
\end{lem}

Combining the lemma of logarithmic derivatives with Lemma \ref{l2}, the key result for meromorphic function $f$ with $\rho_1(f)<1$ is easily obtained
\bea\label{sm} m\left(r, \frac{f^{(k)}(z+c)}{f(z)}\right)\leq m\left(r,\frac{f^{(k)}(z+c)}{f(z+c)}\right)+m\left(r,\frac{f(z+c)}{f(z)}\right)=S(r,f)\eea
for all $r$ outside of a possible exceptional set $E$ with finite logarithmic measure, where $k\in\mathbb{N}$.

\begin{lem}\label{l4}\cite{HYX} Let $f$ be a non-constant meromorphic function with $\rho_1(f)<1$ and $c\in\mathbb{C}\backslash \{0\}$. Then $T(r,f(z+c))=T(r,f)+S(r,f)$ and $N(r,f(z+c))=N(r,f)+S(r,f).$ 
\end{lem}

We define $g(z)=\mathscr{L}_k(f(z+c))$ and the following auxiliary functions
\bea\label{3} \psi=\frac{\delta(g) (f-g)}{(g-a_1)(g-a_2)},\eea
where $\delta(g)=(a_1'-a_2')(g-a_1)-(a_1-a_2)(g'-a_1')=(a_1'-a_2')(g-a_2)-(a_1-a_2)(g'-a_2')$,
\beas\label{3a} H_{nm}=n\phi-m\psi\;\;\text{and}\;\;H=\frac{\delta(f)}{(f-a_1)(f-a_2)}-\frac{\delta(g)}{(g-a_1)(g-a_2)},\;\text{where}\;m,n\in\mathbb{N}.\eeas

\begin{lem}\label{l4} Let $f$ be a non-constant entire function with $\rho_1(f)<1$, $c\in\mathbb{C}\backslash \{0\}$ and $a_1, a_2\in S(f)$ be distinct. If $f$ and $g$ share $a_1, a_2$ IM and $T(r,f)=T(r,g)+S(r,f)$,
then $f\equiv g$.
\end{lem}

\begin{proof} If $f\equiv g$, then the proof is trivial. In the following, we assume that $f\not\equiv g$.
Using second fundamental theorem for small functions (see \cite{KY}), we get $T(r,f)\leq \ol N(r,a_1;f)+\ol N(r,a_2;f)+S(r,f)$. 
Note that $\sum_{i=1}^2\ol N(r,a_i;f)\leq N(r,0;f-g)$ and so from (\ref{sm}), we get
\bea\label{rm} \sideset{}{_{i=1}^2}{\sum}\ol N(r,a_i;f)\leq T(r,f-g)+S(r,f)
\leq m(r,f-g)+S(r,f)\leq T(r,f)+S(r,f).\eea

Therefore
\bea\label{1} T(r,f)= \ol N(r,a_1;f)+\ol N(r,a_2;f)+S(r,f).\eea

If possible, suppose $\delta\left(g\right)\equiv 0$. Then $\frac{g'-a_1'}{g-a_1}\equiv \frac{a_1'-a_2'}{a_1-a_2}$ and so
$g\equiv a_1+c_0(a_1-a_2)$, where $c_0\in\mathbb{C}\backslash \{0\}$. Since $T(r,f)=T(r,g)+S(r,f)$, we have $T(r,f)=S(r,f)$, which is a contradiction. Therefore $\delta\left(g\right)\not\equiv 0$. Similarly $\delta(f)\not\equiv 0$. Consequently $\phi\not\equiv 0$ and $\psi\not\equiv 0$. Also by the given condition, it is easy to get $N(r,\phi)=S(r,f)$. Note that 
\beas m\left(r,\frac{\delta(f)}{(f-a_1)(f-a_2)}\right)&=&m\left(r,\frac{1}{a_1-a_2}\left(\frac{\delta(f)}{f-a_1}-\frac{\delta(f)}{f-a_2}\right)\right)\leq S(r,f).\eeas

On the other hand, we have
\beas \frac{\delta(f)f}{(f-a_1)(f-a_2)}=\frac{\delta(f)}{f-a_1}+\frac{a_2\delta(f)}{(f-a_1)(f-a_2)}\;\;
\text{and so}\;\; m\left(r,\frac{\delta(f)f}{(f-a_1)(f-a_2)}\right)=S(r,f).\eeas

Clearly 
\bea\label{bb.1} \phi=\frac{\delta(f)f}{(f-a_1)(f-a_2)}\left(1-\frac{g}{f}\right)\eea 
and so from (\ref{sm}), we get $m(r,\phi)=S(r,f)$. Therefore $T(r,\phi)=S(r,f)$. Again from (\ref{bb.1}), we get $m(r,0;f)=S(r,f)$.
Let $a_3=a_1+l(a_1-a_2)$, where $l\in\mathbb{N}$.
Then in view of (\ref{1}) and using second fundamental theorem for small functions (see \cite{KY}), we get
\beas 2T(r,f)\leq \ol N(r,a_1;f)+\ol N(r,a_2;f)+\ol N(r,a_3;f)+S(r,f)
\leq 2T(r,f)-m(r,a_3;f)+S(r,f),\eeas
i.e., $m(r,a_3;f)=S(r,f)$. Therefore
\bea\label{5} m(r,0;f)+m(r,a_3;f)=S(r,f).\eea

Also in view of (\ref{1}) and using second fundamental theorem for small functions (see \cite{KY}), we get
\beas 2T(r,g)&\leq&\ol N(r,a_1;g)+\ol N(r,a_2;g)+\ol N(r,a_3;g)+S(r,f)\\
&=& \ol N(r,a_1;f)+\ol N(r,a_2;f)+T(r,g)
-m(r,a_3;g)+S(r,f)\\
&=& T(r,f)+T(r,g)-m(r,a_3;g)+S(r,f)=2T(r,g)-m(r,a_3;g)+S(r,f),\eeas
from which we get
\bea\label{8} m(r,a_3;g)=S(r,f).\eea

Now from (\ref{sm}) and (\ref{5}), we get
\beas m\left(r,\frac{g-a_3}{f-a_3}\right)\leq m\left(r,\frac{\mathscr{L}_k (f(z+c)-a_3(z+c))}{f-a_3}\right)+m(r,a_3;f)+S(r,f)=S(r,f),\eeas
i.e.,
\bea \label{9} m\left(r,\frac{g-a_3}{f-a_3}\right)=S(r,f).\eea

Noting that $T(r,f)=T(r,g)+S(r,f)$ and so using Lemma \ref{l1}, (\ref{5})-(\ref{9}), we get
\bea m\left(r,\frac{f-a_3}{g-a_3}\right)&=& T\left(r,\frac{f-a_3}{g-a_3}\right)-N\left(r,\frac{f-a_3}{g-a_3}\right)\nonumber\\
&=& T\left (r,\frac{g-a_3}{f-a_3}\right)-N\left(r,\frac{f-a_3}{g-a_3}\right)+O(1)\nonumber\\
&=& N\left(r,\frac{g-a_3}{f-a_3}\right)+m\left(r,\frac{g-a_3}{f-a_3}\right)-N\left(r,\frac{f-a_3}{g-a_3}\right)+O(1)\nonumber\\
&=& N\left(r,\frac{g-a_3}{f-a_3}\right)-N\left(r,\frac{f-a_3}{g-a_3}\right)+S(r,f)\nonumber\\
&=& N(r,a_3;f)-N(r,a_3;g)+S(r,f)
=T(r,f)-T(r,g)+S(r,f)=S(r,f),\nonumber\eea
i.e.,
\bea\label{10} m\Big(r,\frac{f-a_3}{g-a_3}\Big)=S(r,f).\eea

Note that
\beas \frac{\delta(g)(g-\alpha)}{(g-a_1)(g-a_2)}=\frac{\delta(g)}{g-a_1}+\frac{(a_2-\alpha)\delta(g)}{(g-a_1)(g-a_2)},\eeas
where $\alpha$ is an arbitrary small function of $f$ and so
\bea\label{x0} m\left(r,\frac{\delta(g)(g-\alpha)}{(g-a_1)(g-a_2)}\right)=S(r,f).\eea

Then from (\ref{10}) and (\ref{x0}), we have
\beas m(r,\psi)&\leq& m\left(r,\frac{\delta(g)(g-a_3)}{(g-a_1)(g-a_2)}\right)+m\left(r,\frac{f-a_3}{g-a_3}\right)+O(1)
\leq S(r,f),\eeas
i.e., $m(r,\psi)=S(r,f)$. Since $f$ and $g$ share $a_1$ and $a_2$ IM, we get $N(r,\psi)=S(r,f)$ and so $T(r,\psi)=S(r,f)$.
We now consider following two cases.\par

{\bf Case 1.} Let $H_{nm}\equiv 0$. Then (\ref{2}) and (\ref{3}) give
$n\big(\frac{f'-a_1'}{f-a_1}-\frac{f-a_2'}{f-a_2}\big)\equiv m\big(\frac{g'-a_1'}{g-a_1}-\frac{g'-a_2'}{g-a_2}\big)$.
On integration, we get $\big(\frac{f-a_1}{f-a_2}\big)^n\equiv c_1\big(\frac{g-a_1}{g-a_2}\big)^m$, $c_1\in\mathbb{C}\backslash \{0\}$. If $n\neq m$, then by Mohon'ko lemma \cite{AZM}, we get
$n\;T(r,f)=m\;T(r,g)+S(r,f)$, which is impossible as $T(r,f)=T(r,g)+S(r,f)$. Hence $n=m$ and so
$\frac{f-a_1}{f-a_2}\equiv c_2 \frac{g-a_1}{g-a_2}$,
where $c_2\in\mathbb{C}\backslash \{0\}$. Since $f\not\equiv g$, we have $c_2\neq 1$ and so $\frac{1-c_2}{c_2}\frac{f-a_4}{f-a_2}\equiv \frac{a_2-a_1}{g-a_2}$,
%\bea\label{14} \frac{1-c_2}{c_2}\frac{f-a_4}{f-a_2}\equiv \frac{a_2-a_1}{\mathscr{L} f(z+c)-a_2},\eea
where $a_4=\frac{a_1-a_2c_2}{1-c_2}$ such that $a_4\not\equiv a_1, a_2$. Since $f$ is a non-constant entire function and $f$, $g$ share $a_2$ IM, we get $N(r,a_4;f)=S(r,f)$. Also we have $F+\frac{c_2}{1-c_2}=\frac{f-a_4}{a_2-a_1}$, where $F=\frac{f-a_1}{a_2-a_1}.$ Then in view of the second fundamental theorem and using (\ref{1}), we get
\beas 2T(r,f)&\leq& \ol N(r,\infty,F) + \ol N(r,0,F)+\ol N(r,1;F)+\ol N\left(r,-\frac{c_2}{1-c_2};F\right)+S(r,f)\\
&=& \ol N(r,a_1;f)+\ol N(r,a_2;f)+S(r,f)=T(r,f)+S(r,f),\eeas 
which is impossible.\par

{\bf Case 2.} Let $H_{nm}\not\equiv 0$ for all $m, n\in\mathbb{N}$. 
Let $z_{m,n}\in S_{(m,n)}(a_1)\cup S_{(m,n)}(a_2)$ such that $a_1(z_{m,n}), a_2(z_{m,n})\not= 0, \infty$ and $a_1(z_{m,n})-a_2(z_{m,n})\neq 0$. 
Now from (\ref{2}) and (\ref{3}), we get
\bea\label{15} H_{nm}=(f-g)\left\lbrack \left(n\frac{f'-a_2'}{f-a_2}-m\frac{g'-a_2'}{g-a_2}\right)-\left(n\frac{f'-a_1'}{f-a_1}-m\frac{g'-a_1'}{g-a_1}\right)\right\rbrack.\eea

Then from (\ref{15}), we get $H_{nm}(z_{m,n})=0$. Therefore
\beas \ol N_{(m,n)}(r,a_1;f)+\ol N_{(m,n)}(r,a_2;f)&\leq&
N(r,0;H_{nm})+S(r,f)\\&\leq& T(r,\phi)+T(r,\psi)+S(r,f)=S(r,f)\eeas
and so from (\ref{1}), we get
\beas T(r,f)&=&\ol N(r,a_1;f)+\ol N(r,a_2;f)+S(r,f)\\
&=& \sideset{}{_{m,n}}{\sum}\left(\ol N_{(m,n)}(r,a_1;f)+\ol N_{(m,n)}(r,a_2;f)\right)+S(r,f)\\
%&=&\sum\limits_{m+n\leq 4}\left(\ol N_{(m,n)}(r,a_1;f)+\ol N_{(m,n)}(r,a_2;f)\right)\\&&
%+\sum\limits_{m+n\geq 5}\left(\ol N_{(m,n)}(r,a_1;f)+\ol N_{(m,n)}(r,a_2;f)\right)+S(r,f)\\
&=&\sideset{}{_{m+n\geq 5}}{\sum}\big(\ol N_{(m,n)}(r,a_1;f)+\ol N_{(m,n)}(r,a_2;f)\big)+S(r,f)\\
&\leq & \frac{1}{5}\big(N(r,a_1;f)+N\left(r,a_1;g\right)
 +N(r,a_2;f)+N\left(r,a_2;g\right)\big)+S(r,f)\\
&\leq& \frac{2}{5}\left(T(r,f)+T(r,g)\right)+S(r,f)
=\frac{4}{5}\;T(r,f)+S(r,f),\eeas
which is impossible.
\end{proof}

\section {{\bf Proofs of the main results}} 

\begin{proof}[{\bf Proof of Theorem \ref{t1}}] 
We know that $f$ and $g$ share $a_1$ and $a_2$ IM, where $a_1, a_2\in S(f)$ are distinct such that either $a_1$ and $a_2$ are $c$-periodic or $a_1'-a_2'$ is non-constant.

Now we divide the following two cases.\par

{\bf Case 1.} Let $\phi\not\equiv 0$, where $\phi$ be defined by (\ref{2}). Clearly $f\not\equiv g$. 
Now from the proof of Lemma \ref{l4}, we have $\delta(f)\not\equiv 0$ and $T(r,\phi)=S(r,f)$. Consequently $\psi\not\equiv 0$, where $\psi$ is defined by (\ref{3}).

Denote by $\ol N\left(r,a_1;f,g\mid \geq 2\right)$ the reduced counting function of multiple $0$-points of $f-a_1$ and $g-a_1$.

Let $z_{p,q}\in S_{(p,q)}(a_1)\;(\text{or}\;S_{(p,q)}(a_2))$ such that $p\geq 2$ and $q\geq 2$. Then from (\ref{2}), one can easily conclude that $z_{p,q}$ is a zero of $\phi$. Also from the proof of Lemma \ref{l4}, we have $T(r,\phi)=S(r,f)$ and so 
\bea\label{e1.1}\label{e1.2}\label{e31.14} \ol N(r,a_1;f,g\mid \geq 2)+\ol N(r,a_2;f,g\mid \geq 2)=S(r,f).\eea

Denote by $N(r,0;f-g\mid f\not=a_1, a_2)$ the counting function of those $0$-points of $f-g$ which are neither the $0$-points of $f-a_1$ nor the $0$-points of $f-a_2$. Also we denote by $\ol N_{(l+1}(r,0;f-g\mid f=a_1,a_2)$ the reduced counting function of those $0$-points of $f-g$ with multiplicity greater than $l$ which are the $0$-points of both $f-a_1$ and $f-a_2$.

Now from (\ref{2}), we can easily deduce that 
\bea\label{e1.3}\label{e1.4}\label{e21.1} \ol N_{(2}(r,0;f-g\mid f= a_1, a_2)+N(r,0;\delta(f)(f-g)\mid f\neq a_1, a_2)=S(r,f).\eea

Rewriting (\ref{2}), we get
\beas f'=\frac{\beta _{1,2}f^2+\beta_{1,1}f+\beta_{1,0}+P_1}{f-g},\eeas
where $\beta_{1,2}=\frac{a_1'-a_2'-\phi}{a_1-a_2}$, $\beta_{1,1}= a_1'-a_1\frac{a_1'-a_2'}{a_1-a_2}+\frac{(a_1+a_2)\phi}{a_1-a_2}$, $\beta_{1,0}= -\frac{\phi a_1 a_2}{a_1-a_2}$ and $P_1=-\frac{a_1'-a_2'}{a_1-a_2}fg -\big(a_1'-a_1\frac{a_1'-a_2'}{a_1-a_2}\big)g$.
Note that
\bea\label{e3.1} f'(z+c)=\frac{\alpha_{1,2}f^2(z+c)+\alpha_{1,1}f(z+c)+\alpha_{1,0}+Q_1}{f(z+c)-g(z+c)},\eea
where $\alpha_{1,j}(z)=\beta_{1,j}(z+c)$, $j=0, 1, 2$ and $Q_1(z)=P_1(z+c)$. 

We now divide following two sub-cases.\par

{\bf Sub-case 1.1.} Let $\phi\not\equiv a_1'-a_2'$. Certainly $\alpha_{1,2}\not\equiv 0$. 
Now by induction and using (\ref{e3.1}) repeatedly, we obtain the following
\bea\label{e3.3} f^{(i)}(z+c)=\frac{\sideset{}{_{j=0}^{2i}}{\sum}\alpha_{i,j}f^j(z+c)+Q_i}{(f(z+c)-g(z+c))^{2i-1}},i=1,2,\ldots,k,\eea
where
\bea\label{e3.3a} Q_i=Q_i\big(f, g,g',\ldots,g ^{(i-1)}\big) 
=\sum\limits_{\substack{l<2i\\ l+j_1+j_2+\ldots+j_k\leq 2i}}\beta_{l, j_1, j_2,\ldots,j_i}f^l(g )^{j_1}(g ')^{j_2}\ldots (g ^{(2i-1)})^{j_i}.\eea

Here $\alpha_{k,j}, \beta_{l, j_1, j_2,\ldots,j_l}\in S(f)$ and $\psi_p:=\alpha_{p,2p}$ satisfies the recurrence formula
\bea\label{e3.4} \psi_1=\alpha_{1,2},\; \psi_{p+1}=\psi_p'+\psi_1 \psi_p,\;p=1,2,\ldots,k-1.\eea

Now using (\ref{e3.3}), we get
\bea\label{e3.3aa} g(z)=\frac{\sideset{}{_{j=0}^{2k}}{\sum}\gamma_j(z+c) f^j(z+c)+Q(z+c)}{(f(z+c)-g(z+c))^{2k-1}},\eea
where $\gamma_j\in S(f)$, $j=1,2,\ldots,k$ such that
\bea\label{e7} \gamma_{2k}(z+c)&=&b_0(z+c)+\sideset{}{_{i=1}^{k}}{\sum} b_i(z+c) \psi_i(z+c)\;\text{and}\nonumber\\ 
Q&=&\sum\limits_{\substack{l<2i\\ l+j_1+j_2+\ldots+j_k\leq 2k}}\zeta_{l, j_1, j_2,\ldots,j_k}f^l(g )^{j_1}(g ')^{j_2}\ldots (g ^{(i-1)})^{j_k}.\eea

Clearly $\zeta_{l, j_1, j_2,\ldots,j_k}\in S(f)$. If $\gamma_{2k}\equiv 0$, then the proof is done. Next suppose $\gamma_{2k}\not\equiv 0$. Then from (\ref{e3.3aa}), we get
\bea\label{e3.6a} \sideset{}{_{j=0}^{2k}}{\sum}\gamma_j f^j(z+c)=g(z)\left(f(z+c)-g(z+c)\right)^{2k-1}-Q(z+c).\eea

Using (\ref{sm}), (\ref{5}), (\ref{e7}) and the lemma of the logarithmic derivative, we deduce that
\bea\label{e3.5} m\left(r,\frac{Q(z+c)}{f^{2k-1}(z+c)g}\right)=S(r,f).\eea

Now from (\ref{sm}), (\ref{e3.6a}) and (\ref{e3.5}), we get 
\beas &&2k T(r,f(z+c))\\&=&
T\left(r,\sideset{}{_{j=0}^{2k}}{\sum}\gamma_j f^j(z+c)\right)+S(r,f)\\
&\leq& m\left(r, \left(1-\frac{g(z+c)}{f(z+c)}\right)^{2k-1}-\frac{Q(z+c)}{f^{2k-1}(z+c)g}\right)+m\big(r,f^{2k-1}(z+c)g\big)+S(r,f)\\
&\leq& m\big(r, f^{2k-1}(z+c)g\big)+S(r,f)
%&\leq& (2k-1) m\left(r, 1-\frac{g }{g}\right)+m\left(r, \frac{Q_k}{g^{2k-1}g^{(k)}}\right)
%+m(r,g^{2k-1})+m(r, g )\\&&+m\left(r,\frac{g^{(k)}}{g }\right)+S(r,f)
%&\leq& m\left(r, \left(1-\frac{ g }{f(z)}\frac{f(z)}{f(z+c)}\right)^{2k-1}\right)+m\left(r, \frac{Q_k}{f^{2k-1}(z+c)g (z+c)}\right)\\
%&&+m(r,f^{2k-1}(z+c))+m\left(r,\frac{g (z+c)}{ g } g \right)+S(r,f)\\
\leq (2k-1)T(r,f(z+c))+T(r,g(z))+S(r,f),\eeas
i.e., $T(r,f(z+c))\leq T(r,g)+S(r,f)$ and so by Lemma \ref{l2}, we get $T(r,f)\leq T(r,g)+S(r,f)$. 
Again (\ref{sm}) yields
\bea\label{bb.2} T(r,g)=m(r,g)\leq m(r,f)+S(r,f)=T(r,f)+S(r,f).\eea

Hence $T(r,f)=T(r,g)+S(r,f)$ and so from Lemma \ref{l4}, we get $f\equiv g$, which is absurd.\par

{\bf Sub-case 1.2.} Let $\phi\equiv a_1'-a_2'$. If $a_1'-a_2'$ is not constant, then
$a_1''\not\equiv a_2''$. Next we suppose $a_1$ and $a_2$ are $c$-periodic. We claim that
$a_1''\not\equiv a_2''$. If not, suppose $a_1''\equiv a_2''$. On integration, we get $a_1-a_2=c_0 z+c_1$, where $c_0, c_1\in\mathbb{C}$ such that $(c_0,c_1)\neq (0,0)$. Since $a_1$ and $a_2$ are $c$-periodic, $a_1-a_2$ is also $c$-periodic.
As $a_1-a_2$ is a polynomial, we see that $a_1-a_2$ is a constant. Then $\phi\equiv 0$, which is absurd. Hence  $a_1''\not\equiv a_2''$.  Differentiating (\ref{2}) once, we get
\bea\label{4a} &&((a_1''-a_2'')(f-a_1)-(a_1-a_2)(f''-a_1''))(f-g)\\&&
+((a_1'-a_2')(f-a_1)-(a_1-a_2)(f'-a_1'))(f'-g')\nonumber\\&=&
\phi'(f-a_1)(f-a_2)+\phi(f'-a_1')(f-a_2)+\phi(f-a_1)(f'-a_2'),\;\text{i.e.,}\nonumber\eea
\bea\label{e31.33} && (a_1''-a_2'')(f-a_1)^2-(a_1''-a_2'')(f-a_1)(g -a_1)
-(a_1-a_2)(f-a_1)(g''-a_1'')\\&&+(a_1-a_2)(g''-a_1'')(g -a_1)
+(a_1'-a_2')(f-a_1)(f'-a_1')\nonumber\\&&-(a_1'-a_2')(f-a_1)((g '-a_1')
-(a_1-a_2)(f'-a_1')^2+(a_1-a_2)(f'-a_1')((g '-a_1')\nonumber\\
&=&(a_1''-a_2'')(f-a_1)(f-a_2)+(a_1'-a_2')(f'-a_1')(f-a_2)
+(a_1'-a_2')(f-a_1)(f'-a_2').\nonumber\eea

Let $z_{p,1}\in S_{(p,1)}(a_1)(p \geq 2)$ such that $a_1(z_{p,1})-a_2(z_{p,1})\not=0,\infty$ and $a_1'(z_{p,1})-a_2'(z_{p,1})\not=0$.  
Then in some neighbourhood of $z_{p,1}$, we get by Taylor's expansion
\bea\label{sm1} \left\{\begin{array}{clcr} f(z)-a_1(z)&=&\tilde b_{p}(z-z_{p,1})^{p}+\tilde b_{p+1}(z-z_{p,1})^{p+1}+\ldots (\tilde b_{p}\not=0),\\
g(z) -a_1(z)&=&c_{1}(z-z_{p,1})+c_{2}(z-z_{p,1})^{2}+\ldots (c_{1}\not=0),\\
\phi(z)&=&d_{0}+d_{1}(z-z_{p,1})+d_{2}(z-z_{p,1})^2+\ldots\;\;(d_{0}\neq 0).\end{array}\right.\eea
Clearly 
\bea\label{sm2}&& \left\{\begin{array}{clcr} f'(z)-a_1'(z)&=&p\tilde b_{p}(z-z_{p,1})^{p-1}+(p+1)\tilde b_{p+1}(z-z_{p,1})^{p}+\ldots,\\
f''(z)-a_1''(z)&=&p(p-1)\tilde b_{p}(z-z_{p,1})^{p-2}+p(p+1)\tilde b_{p+1}(z-z_{p,1})^{p-1}+\ldots,\\
g'(z)-a_1'(z)&=&c_{1}+2c_{2}(z-z_{p,1})+\ldots,\\
g''(z)-a_1''(z)&=&2c_{2}+6c_{3}(z-z_{p,1})+\ldots,\\
\phi'(z)&=&d_{1}+2d_{2}(z-z_{p,1})+3d_{2}(z-z_{p,1})^2\ldots,\\
\phi''(z)&=&2d_{2}+6d_{2}(z-z_{p,1})\ldots.\end{array}\right.\eea

Now from (\ref{4a}), (\ref{sm1}) and (\ref{sm2}), we get
\bea\label{sm3} d_{0}=p c_{1},\;\;\text{i.e.},\;\;g'(z_{p,1})-a_1'(z_{p,1})=c_1=\frac{\phi(z_{p,1})}{p}=\frac{a_1'(z_{p,1})-a_2'(z_{p,1})}{p}.\eea
   
It is easy to calculate, from (\ref{e31.33})-(\ref{sm3}) that
\beas\label{e31.34} && \Big(\tilde b_p c_2 p^2 (p+1)(a_1(z_{p,1})-a_2(z_{p,1}))-\tilde b_p(a_1'(z_{p,1})-a_2'(z_{p,1}))^2\\&&+\tilde b_{p+1}(p+1)^2(a_1(z_{p,1})-a_2(z_{p,1}))(a_1'(z_{p,1})-a_2'(z_{p,1}))\\
&&-p\tilde b_p\left( (a_1'' (z_{p,1})-a_2''(z_{p,1}))(a_1(z_{p,1})-a_2(z_{p,1}))+(a_1'(z_{p,1})-a_2'(z_{p,1}))^2 \right)\\
&&- \tilde b_{p+1} p(p+1) (a_1(z_{p,1})-a_2(z_{p,1}))(a_1'(z_{p,1})-a_2'(z_{p,1}))\Big) (z-z_{p,1})^p\\&&+O\big((z-z_{p,1})^{p+1}\big)\equiv 0,\eeas
which shows that 
\bea\label{e31.35} &&(\tilde b_p c_2 p^2 (p+1)(a_1(z_{p,1})-a_2(z_{p,1}))-\tilde b_p(a_1'(z_{p,1})-a_2'(z_{p,1}))^2\nonumber\\
&&+\tilde b_{p+1}(p+1)^2(a_1(z_{p,1})-a_2(z_{p,1}))(a_1'(z_{p,1})-a_2'(z_{p,1}))\nonumber\\
&&-p\tilde b_p\left( (a_1'' (z_{p,1})-a_2''(z_{p,1}))(a_1(z_{p,1})-a_2(z_{p,1}))+(a_1'(z_{p,1})-a_2'(z_{p,1}))^2\right)\nonumber\\
&&- \tilde b_{p+1} p(p+1) (a_1(z_{p,1})-a_2(z_{p,1}))(a_1'(z_{p,1})-a_2'(z_{p,1}))=0.\eea

Differentiating (\ref{4a}), once we get
\beas\label{4aa} &&\left((a_1'''-a_2''')(f-a_1)+(a_1''-a_2'')(f'-a_1')-(a_1'-a_2')(g''-a_1'')-(a_1-a_2)(g'''-a_1''')\right)\times\nonumber\\&&(f-g)
+2((a_1''-a_2'')(f-a_1)-(a_1-a_2)(g''-a_1''))(f'-g')\nonumber\\&&
+((a_1'-a_2')(f-a_1)-(a_1-a_2)(f'-a_1'))(f''-g'')\nonumber\\&=&
\phi''(f-a_1)(f-a_2)+2\phi'(f'-a_1')(f-a_2)
+2\phi'(f-a_1)(f'-a_2')\nonumber\\&&+\phi(f''-a_1'')(f-a_2)+2\phi(f'-a_1')(f'-a_2')+\phi(f-a_1)(f''-a_2'')\eeas
and so
\begin{small}
\bea\label{e31.36} &&(a_1'''-a_2''')(f-a_1)^2-(a_1'''-a_2''')(f-a_1)(g -a_1)
+(a_1''-a_2'')(f-a_1)(f'-a_1')\nonumber\\&&-(a_1''-a_2'')(f'-a_1')(g -a_1)
-(a_1'-a_2')(f-a_1)(g''-a_1'')+(a_1'-a_2')(g''-a_1'')(g -a_1)\nonumber\\&&
-(a_1-a_2)(f-a_1)(g'''-a_1''')+(a_1-a_2)(g'''-a_1''')(g -a_1)
+2(a_1''-a_2'')(f-a_1)(f'-a_1')\nonumber\\&&-2(a_1''-a_2'')(f-a_1)((g '-a_1')
-2(a_1-a_2)(f'-a_1')(g''-a_1'')+(a_1'-a_2')(f-a_1)(g''-a_1'')\nonumber\\&&+2(a_1-a_2)(g''-a_1'')((g '-a_1')
-(a_1'-a_2')(f-a_1)((g ''-a_1'')\nonumber\\
&&-(a_1-a_2)(f'-a_1')(g''-a_1'')+(a_1-a_2)(f'-a_1')((g ''-a_1'')\nonumber\\
&=&(a_1'''-a_2''')(f-a_1)(f-a_2)+2(a_1''-a_2'')(f'-a_1')(f-a_2)
+2(a_1''-a_2'')(f-a_1)(f'-a_2')\nonumber\\&&+(a_1'-a_2')(g''-a_1'')(f-a_2)
+2(a_1'-a_2')(f'-a_1')(f'-a_2')+(a_1'-a_2')(f-a_1)(g''-a_2'').\eea
\end{small}

It is easy to calculate, from (\ref{sm1})-(\ref{sm3}) and (\ref{e31.36}) that
\beas && \Big( \tilde b_p c_2 p^2 (p+1)(a_1(z_{p,1})-a_2(z_{p,1}))+\tilde b_p(p-1)(a_1'(z_{p,1})-a_2'(z_{p,1}))^2\\&&+\tilde b_{p+1}(p+1)^2(a_1(z_{p,1})-a_2(z_{p,1}))(a_1'(z_{p,1})-a_2'(z_{p,1}))\\
&&-2p\tilde b_p\left( (a_1'' (z_{p,1})-a_2''(z_{p,1}))(a_1(z_{p,1})-a_2(z_{p,1}))+(a_1'(z_{p,1})-a_2'(z_{p,1}))^2\right)\\
&&- \tilde b_{p+1} p(p+1) (a_1(z_{p,1})-a_2(z_{p,1}))(a_1'(z_{p,1})-a_2'(z_{p,1}))\Big) (z-z_{p,1})^{p-1}\\&&+O\left(\left(z-z_{p,1}\right)^p\right)\equiv 0,\eeas
which shows that
\bea\label{e31.38} &&\tilde b_p c_2 p^2 (p+1)(a_1(z_{p,1})-a_2(z_{p,1}))+\tilde b_p(p-1)(a_1'(z_{p,1})-a_2'(z_{p,1}))^2\nonumber\\
&&+\tilde b_{p+1}(p+1)^2(a_1(z_{p,1})-a_2(z_{p,1}))(a_1'(z_{p,1})-a_2'(z_{p,1}))\nonumber\\
&&-2p\tilde b_p\left( (a_1'' (z_{p,1})-a_2''(z_{p,1}))(a_1(z_{p,1})-a_2(z_{p,1}))+(a_1'(z_{p,1})-a_2'(z_{p,1}))^2\right)\nonumber\\
&&- \tilde b_{p+1} p(p+1) (a_1(z_{p,1})-a_2(z_{p,1}))(a_1'(z_{p,1})-a_2'(z_{p,1}))=0.\eea

Now from (\ref{e31.35}) and (\ref{e31.38}), we have
\bea\label{e31.39} \tilde b_p p (a_1(z_{p,1})-a_2(z_{p,1}))(a_1''(z_{p,1})-a_2''(z_{p,1}))=0.\eea

Since $\tilde b_p\not=0$ and $a_1(z_{p,1})-a_2(z_{p,1})\not=0,\infty$, from (\ref{e31.39}), we get $a_1''(z_{p,1})-a_2''(z_{p,1})=0$.
As $a_1''\not\equiv a_2''$, we deduce that
\bea\label{rm.1}\sideset{}{_{p\geq 2}}{\sum}\ol N_{(p,1)}(r,a_1;f)\leq N(r,0;a_1''-a_2'')\leq S(r,f).\eea

Let $\hat{z}_{p,1}\in S_{(p,1)}(a_2)(p \geq 2)$ such that $\phi(\hat{z}_{p,1})\neq 0, \infty$ and $a_1(\hat{z}_{p,1})-a_2(\hat{z}_{p,1})\neq 0, \infty$. 
Then in some neighbourhood of $\hat{z}_{p,1}$, we get by Taylor's expansion
\bea\label{sm5a} \left\{\begin{array}{clcr} f(z)-a_2(z)&=&\hat{b}_{p}(z-\hat{z}_{p,1})^{p}+\hat{b}_{p+1}(z-\hat{z}_{p,1})^{p+1}+\ldots (\hat{b}_{p}\not=0),\\
g(z) -a_2(z)&=&\hat{c}_{1}(z-\hat{z}_{p,1})^q+\hat{c}_{2}(z-\hat{z}_{p,1})^{2}+\ldots (\hat{c}_{1}\not=0),\\
\phi(z)&=&\hat{d}_{0}+\hat{d}_{1}(z-\hat{z}_{p,q})+\hat{d}_{2}(z-\hat{z}_{p,q})^2+\ldots\;\;(\hat{d}_{0}\neq 0).\end{array}\right.\eea

Similarly, we get
\beas g'(\hat{z}_{p,1})-a_2'(\hat{z}_{p,1})=\hat{c}_1=-\frac{\phi(\hat{z}_{p,1})}{p}=-\frac{a_1'(\hat{z}_{p,1})-a_2'(\hat{z}_{p,1})}{p}.\eeas

Now proceeding in the same way as done above and using (\ref{sm5a}) instead of (\ref{sm1}), we can conclude that $a_1''(\hat{z}_{p,1})-a_2''(\hat{z}_{p,1})=0$ and so 
\bea\label{rm.2}\sideset{}{_{p\geq 2}}{\sum}\ol N_{(p,1)}(r,a_2;f)\leq N(r,0;a_1''-a_2'')\leq S(r,f).\eea
 
Therefore from (\ref{e31.14}), (\ref{rm.1}) and (\ref{rm.2}), we get
\bea\label{rmbm.1} \ol N_{(2}(r,a_1;f)+\ol N_{(2}(r,a_2;f)\leq\sideset{}{_{p\geq 2}}{\sum}(\ol N_{(p,1)}(r,a_1;f)+\ol N_{(p,1)}(r,a_2;f))=S(r,f).\eea

On the other hand from (\ref{2}), we get $f-g=\frac{(a_1'-a_2') (f-a_1)(f-a_2)}{\delta(f)}$ and so
\bea\label{e31.2}\label{e31.3} g -a_1
=\frac{-(a_1-a_2)(f-a_1)(f'-a_2')}{\delta(f)}\;\text{and}\;g -a_2=\frac{-(a_1-a_2)(f-a_2)(f'-a_1')}{\delta(f)}.\eea

Let
\bea\label{bm.7} \varphi_1=\frac{f-g }{(f-a_1)(f-a_2)}.\eea

Clearly $\varphi_1\not\equiv 0$. Since $\phi=a_1'-a_2'$, from (\ref{2}) and (\ref{bm.7}), we get
\bea\label{bm.8} (a_1'-a_2')\left(f-a_1-\frac{1}{\varphi_1}\right)=(a_1-a_2)(f'-a_1').\eea

We claim that $\varphi_1$ is non-constant. If not, suppose $\varphi_1=d\in\mathbb{C}\backslash \{0\}$. Then from (\ref{bm.8}), we have 
\[f'-a_1'-\frac{a_1'-a_2'}{a_1-a_2}(f-a_1)=-d\frac{a_1'-a_2'}{a_1-a_2}.\]

Performing integration, we conclude that $f-a_1$ is a small function of $f$, which is impossible.

Let $\phi_1=\frac{\varphi_1'}{\varphi_1}$. Clearly $\phi_1\not\equiv 0$.
Let $z_{1,q}\in S_{(1,q)}(a_1)$ such that $a_1(z_{1,q})-a_2(z_{1,q})\not=0,\infty$. Then $z_{1,q}$ is a zero of $f-g $ and so $\varphi_1(z_{1,q})\neq\infty$.
Similarly if $\hat{z}_{1,q}\in S_{(1,q)}(a_2)$ such that $a_1(\hat{z}_{1,q})-a_2(\hat{z}_{1,q})\not=0,\infty$, then $\varphi_1(\hat{z}_{1,q})\neq\infty$. Now (\ref{e31.14}), (\ref{e21.1}) and (\ref{rmbm.1}) give
$\ol N(r,0;\varphi_1)+N(r,\varphi_1)=S(r,f)$ and so $T(r,\phi_1)=S(r,f)$.
Taking logarithmic differentiation on (\ref{bm.7}), we get
\beas \phi_1=\frac{f'-g'}{f-g}-\frac{f'-a_1'}{f-a_1}-\frac{f'-a_2'}{f-a_2},\;\text{i.e.},\;\eeas
\bea\label{bm.10}&&\phi_1(f-a_1)^2(f-a_2)-\phi_1(f-a_1)(f-a_2)(g -a_1)\\
&=&-(f-a_1)(f-a_2)((g '-a_1')+(f'-a_1')(f-a_2)(g -a_1)\nonumber\\&&+(f-a_1)(f'-a_2')(g -a_1)-(f-a_1)^2(f'-a_2')\nonumber\eea
or 
\bea\label{bm.11}&&\phi_1(f-a_1)(f-a_2)^2-\phi_1(f-a_1)(f-a_2)(g -a_2)\\
&=& -(f-a_1)(f-a_2)((g '-a_2')+(f'-a_1')(f-a_2)(g -a_2)\nonumber\\&&+(f-a_1)(f'-a_2')(g -a_2)-(f'-a_1')(f-a_2)^2.\nonumber\eea

We use $N_{(l+1}(r,a;f)$ to denote the counting function of $a$-points of $f$ with multiplicity greater than $l\in\mathbb{N}$. Also $\ol N_{(l+1}(r,a;f)$ is the reduced counting function.

Let $z_{1,q}\in S_{(1,q)}(a_1)\;(q\geq 3)$ such that $a_1(z_{1,q})-a_2(z_{1,q})\neq 0,\infty$ and $a_1'(z_{1,q})-a_2'(z_{1,q})\neq 0$. 
Clearly $\delta(f(z_{1,q}))\neq 0$ and so from (\ref{e31.2}), we see that $z_{1,q}$ is a zero of $f'-a_2'$ of multiplicity $q-1$. Then from (\ref{bm.10}), we calculate that $\phi_1(z_{1,q})=0$. Since $\phi_1\not\equiv 0$, we have
\beas\label{bm.18} \sideset{}{_{q\geq 3}}{\sum} \ol N_{(1,q)}(r,a_1;f)\leq N(r,0;\phi_1)\leq S(r,f)\eeas
and so from (\ref{e31.14}), we get $\ol N_{(3}(r,a_1;g )=S(r,f).$
Similarly if $\hat{z}_{1,q}\in S_{(1,q)}(a_2)\;(q\geq 3)$ such that $a_1(\hat{z}_{1,q})-a_2(\hat{z}_{1,q})\neq 0,\infty$ and $a_1'(\hat{z}_{1,q})-a_2'(\hat{z}_{1,q})\neq 0$, 
then $\delta(f(\hat{z}_{1,q}))\neq 0$ and $\hat{z}_{1,q}$ is a zero of $f'-a_1'$ of multiplicity $q-1$ by (\ref{e31.3}). 
Also from (\ref{bm.11}), we get $\phi_1(\hat{z}_{1,q})=0$ and so $\sum_{q\geq 3} \ol N_{(1,q)}(r,a_2;f)\leq S(r,f)$. Again
from (\ref{e31.14}), we have $\ol N_{(3}(r,a_2;g)=S(r,f).$ Therefore
\bea\label{bm.12} \ol N_{(3}(r,a_1;g)+\ol N_{(3}(r,a_2;g)=S(r,f).\eea

Let $\alpha$ be an arbitrary small function of $f$. Then from (\ref{sm}), we get
\bea\label{pp1}m\left(r,\frac{g -\alpha}{f-\alpha}\right)=m\left(r,\frac{\mathscr{L}_k(f(z+c)-\alpha(z+c))+\mathscr{L}_k(\alpha(z+c))-\alpha}{f-\alpha}\right)\leq m\left(r,\frac{1}{f-\alpha}\right)+S(r,f).\eea

Now using Lemma \ref{l1}, we get from (\ref{3}), (\ref{x0}) and (\ref{pp1}) that
\beas m(r,\psi)&=&m\left(r,\frac{\delta(g )(g -\alpha)}{(g -a_1)(g -a_2)}\left(\frac{f-\alpha}{g -\alpha}-1\right)\right)\\
&\leq& m\left(r,\frac{\delta(g )(g -\alpha)}{(g -a_1)(g -a_2)}\right)+m\left(r,\frac{f-\alpha}{g -\alpha}-1\right)\\
&\leq& m\left(r,\frac{f-\alpha}{g -\alpha}\right)+S(r,f)
=T\left(r,\frac{g -\alpha}{f-\alpha}\right)-N\left(r,\frac{f-\alpha}{g -\alpha}\right)+S(r,f)\\
&=& m\left(r,\frac{g -\alpha}{f-\alpha}\right)+N\left(r,\frac{g -\alpha}{f-\alpha}\right)-N\left(r,\frac{f-\alpha}{g -\alpha}\right)+S(r,f)\\
&\leq& m\left(r,\frac{1}{f-\alpha}\right)+N(r,0;f-\alpha)-N(r,0;g -\alpha)+S(r,f)\\
%&\leq& m(r,\alpha;g)+N(r,\alpha;g)-N(r,\alpha;g )+S(r,f)\\
&=& T(r,f)-N(r,\alpha;g)+S(r,f).\eeas

Since $N(r,\psi)=S(r,f)$, we have
\bea\label{bm.3} T(r,\psi)\leq T(r,f)-N(r,\alpha;g)+S(r,f).\eea

We use $N(r,a;f\mid=l)$ to denote the counting function of $a$-points of $f$ with multiplicity exactly $l\in\mathbb{N}$. Also $\ol N(r,a;f\mid=l)$ is the reduced counting function.

Now we divide following two sub-cases.\par

{\bf Sub-case 1.2.1.} Let $\ol N(r,a_1;g\mid=2)=S(r,f)$.
Then from (\ref{e31.14}) and (\ref{bm.12}), we get
\beas\label{bm.1}\label{bm.5} \sideset{}{_{q\geq 2}}{\sum}\ol N_{(1,q)}(r,a_1;f)=S(r,f)\eeas
and so from (\ref{rmbm.1}), we get
\bea\label{bm.6} \ol N(r,a_1;f)&=&\sideset{}{_{p,q}}{\sum}\ol N_{(p,q)}(r,a_1;f)=
\ol N_{(1,1)}(r,a_1;f)+\sideset{}{_{p\geq 2}}{\sum}\ol N_{(p,1)}(r,a_1;f)\\&&+\sideset{}{_{q\geq 2}}{\sum}\ol N_{(1,q)}(r,a_1;f)=\ol N_{(1,1)}(r,a_1;f)+S(r,f)\nonumber.\eea

First suppose $H\equiv 0$. On integration, we get
$\frac{f-a_1}{f-a_2} \equiv c_2 \frac{g -a_1}{g -a_2}$,
where $c_2\in\mathbb{C}\backslash \{0\}$. Using Mohon'ko lemma \cite{AZM} we get $T(r,f)=T(r,g )+S(r,f)$ and so Lemma \ref{l4} gives $f\equiv g$, which is impossible.\par

Next suppose $H\not\equiv 0$. Clearly $m(r,H)=S(r,f)$ and $N(r,H)=\ol N_{(1,2)}(r,a_2;f)$. Therefore
\bea\label{bm.14} T(r,H)=m(r,H)+N(r,H)=\ol N_{(1,2)}(r,a_2;f)+S(r,f).\eea

Now from (\ref{rm}) and (\ref{1}), we get $T(r,f)=m(r,f-g )+S(r,f)$.
Therefore from (\ref{2}), (\ref{3}), (\ref{bm.3}) and (\ref{bm.14}), we have
\bea T(r,f)&=& m\left(r,\frac{H (f-g )}{H}\right)+S(r,f)
=m\left(r,\frac{\phi-\psi}{H}\right)+S(r,f)\nonumber\\
&=& T\left(r, \frac{H}{\phi-\psi}\right)+S(r,f)
\leq T(r,\psi)+T(r,H)+S(r,f)\nonumber\\
&\leq& T(r,f)+\ol N_{(1,2)}(r,a_2;f)-N(r,\alpha;g)+S(r,f)\nonumber\eea
and so 
\bea\label{bm.15} N(r,\alpha;g)\leq \ol N_{(1,2)}(r,a_2;f)+S(r,f).\eea

Suppose $\alpha=a_2$. Then from (\ref{bm.15}), we get
\bea\label{bm.16} \ol N(r,a_2;g)\leq \ol N_{(1,2)}(r,a_2;f)+S(r,f).\eea

Since $f$ and $g $ share $a_2$ IM, from (\ref{e31.14}), (\ref{rmbm.1}) and (\ref{bm.12}), we get
\beas \ol N(r,a_2;g)+S(r,f)
=\ol N(r,a_2;f)+S(r,f)=\sideset{}{_{i=1}^2}{\sum}\ol N_{(1,i)}(r,a_2;f)+S(r,f)\eeas
and so from (\ref{bm.16}), we get $\ol N_{(1,1)}(r,a_2;f)=S(r,f)$.
Let 
\bea\label{bm.17} G=\frac{g '-a_1'}{g -a_1}-\frac{f'-a_1'}{f-a_1}-\frac{a_1'-a_2'}{a_1-a_2}.\eea

Clearly $G\not\equiv 0$. Also from (\ref{bm.6}), we get $N(r,G)=S(r,f)$. Since $m(r,G)=S(r,f)$, we have $T(r,G)=S(r,f)$. 
Let $\hat{z}_{1,2}\in S_{(1,2)}(a_2)$ such that $a_1(\hat{z}_{1,2})-a_2(\hat{z}_{1,2})\neq 0, \infty$, $a_1'(\hat{z}_{1,2})-a_2'(\hat{z}_{1,2})\neq 0$ and $\phi(\hat{z}_{1,2})\not=0,\infty$. 
Then $g (\hat{z}_{1,2})=a_2(\hat{z}_{1,2})$ and $g '(\hat{z}_{1,2})=a_2'(\hat{z}_{1,2})$. 
Also from (\ref{e31.3}), we see that $\hat{z}_{1,2}$ is a simple zero of $f'-a_1'$, i.e., $f'(\hat{z}_{1,2})=a_1'(\hat{z}_{1,2})$.
Then from (\ref{bm.17}), we get $G(\hat{z}_{1,2})=0$ and so 
$\ol N_{(1,2)}(r,a_2;f)\leq N(r,0;G)+S(r,f)\leq T(r,G)+S(r,f)=S(r,f)$.
Therefore from (\ref{bm.12}), we get
\beas\label{bm.1a} \sideset{}{_{q\geq 2}}{\sum}\ol N_{(1,q)}(r,a_2;f)=S(r,f)\eeas
and so from (\ref{rmbm.1}), we get
\bea\label{bm.6a} \ol N(r,a_2;f)&=&\ol N_{(1,1)}(r,a_2;f)+\sideset{}{_{p\geq 2}}{\sum}\ol N_{(p,1)}(r,a_2;f)+\sideset{}{_{q\geq 2}}{\sum}\ol N_{(1,q)}(r,a_2;f)\\&=&\ol N_{(1,1)}(r,a_2;f)+S(r,f)\nonumber.\eea

Then from (\ref{1}), (\ref{bm.6}) and (\ref{bm.6a}), we get
\bea\label{bm.2} T(r,f)=\ol N(r,a_1;f)+\ol N(r,a_2;f)+S(r,f)=
%\sum_{i=1}^2\ol N_{(1,1)}(r,a_i;f)+\sum\limits_{p\geq 2}(\ol N_{(p,1)}(r,a_1;f)+\ol N_{(p,1)}(r,a_2;f))\nonumber\\&&
%+\sum\limits_{q\geq 2}(\ol N_{(1,q)}(r,a_1;f)+\ol N_{(1,q)}(r,a_2;f))+S(r,f)\nonumber\\&=&
\ol N_{(1,1)}(r,a_1;f)+\ol N_{(1,1)}(r,a_2;f)+S(r,f).\eea

First suppose $H_{11}\equiv 0$. On integration, we get $\frac{f-a_1}{f-a_2}\equiv c_1\frac{g-a_1}{g-a_2}$, where $c_1\in\mathbb{C}\backslash \{0\}$. Using Mohon'ko lemma \cite{AZM} we get $T(r,f)=T(r,g )+S(r,f)$ and so Lemma \ref{l4} gives $f\equiv g$, which is impossible.\par

Next suppose $H_{11}\not\equiv 0$. Let $z_{1,1}\in S_{(1,1)}(a_1)\cup S_{(1,1)}(a_2)$. Clearly  $H_{11}(z_{1,1})=0$ and so
\bea\label{bm.4}\sideset{}{_{i=1}^2}{\sum}\ol N_{(1,1)}(r,a_i;f)\leq N(r,0;H_{11})+S(r,f)\leq T(r, H_{11})+S(r,f)\leq T(r,\psi)+S(r,f)\eea 
and so from (\ref{bm.3}) and (\ref{bm.2}), we get $T(r,f)\leq T(r,\psi)+S(r,f)\leq T(r,f)-N(r,\alpha;g )+S(r,f)$, i.e., $N(r,\alpha;g )=S(r,f)$. 
%\beas T(r,f)\leq T(r,\psi)+S(r,f)\leq T(r,f)-N(r,\alpha;g )+S(r,f),\;\;
%\text{i.e.},\;\; N(r,\alpha;g )=S(r,f),\eeas
In particular, we have $\ol N(r,a_1;g)+\ol N(r,a_2;g)=S(r,f)$, i.e., $\ol N(r,a_1;f)+\ol N(r,a_2;f)=S(r,f)$ and so from (\ref{1}), we get a contradiction.\par

{\bf Sub-case 1.2.2.} Let $\ol N(r,a_1;g\mid=2)\not=S(r,f)$. If $\ol N(r,a_2;g\mid=2)=S(r,f)$, then proceeding in the way as done in Sub-case 1.2.1, we get a contradiction. Hence we assume  
$\ol N(r,a_2;g\mid=2)\neq S(r,f)$.
In this case, from (\ref{1}), (\ref{rmbm.1}) and (\ref{bm.12}), we get
\bea\label{alr.1} T(r,f)=\sideset{}{_{i=1}^2}{\sum}\left(\ol N_{(1,1)}(r,a_i;f)+\ol N_{(1,2)}(r,a_i;f)\right)+S(r,f).\eea

Clearly $H_{11}\not\equiv 0$, otherwise we get $T(r,g)=T(r,f)+S(r,f)$ and so by Lemma \ref{l4}, we get a contradiction. Let $z_{1,1}\in S_{(1,1)}(a_1)\cup S_{(1,1)}(a_2)$. Then from (\ref{bm.3}) and (\ref{bm.4}), we see that
\beas\sideset{}{_{i=1}^2}{\sum} \ol N_{(1,1)}(r,a_i;f)\leq T(r,f)-N(r,\alpha;g)+S(r,f)\eeas
and so from (\ref{alr.1}), we get 
\bea\label{alr.2} \ol N(r,\alpha;g)\leq N(r,\alpha;g)\leq \sideset{}{_{i=1}^2}{\sum} \ol N_{(1,2)}(r,a_i;f)+S(r,f).\eea

Suppose $\alpha=a_1$. Now from (\ref{rmbm.1}) and (\ref{bm.12}), we have $\ol N(r,a_1;g)=\ol N(r,a_1;f)=\ol N_{(1,1)}(r,a_1;f)+\ol N_{(1,2)}(r,a_1;f)+S(r,f)$. Therefore from (\ref{alr.2}), we get
$\ol N_{(1,1)}(r,a_1;f)\leq N_{(1,2)}(r,a_2;f)+S(r,f)$. Again if we take $\alpha=a_2$, then from (\ref{alr.2}), we get $\ol N_{(1,1)}(r,a_2;f)\leq N_{(1,2)}(r,a_1;f)+S(r,f)$.
Consequently from (\ref{alr.1}), we have
\bea\label{alr.3} T(r,f)\leq 2\sideset{}{_{i=1}^2}{\sum} \ol N_{(1,2)}(r,a_i;f)+S(r,f).\eea

First suppose $H_{21}\equiv 0$. On integration, we have
$\left(\frac{f-a_1}{f-a_2}\right)^2=c_1\frac{g -a_1}{g -a_2}$,
where $c_1\in\mathbb{C}\backslash \{0\}$. Using Mohon'ko lemma \cite{AZM}, we get $2T(r,f)=T(r,g )+S(r,f)$. Then from (\ref{bb.2}), we have $2T(r,f)\leq T(r,f)+S(r,f)$, which is impossible.\par

Next suppose $H_{21}\not\equiv 0$. Let $z_{1,2}\in S_{(1,2)}(a_1)\cup S_{(1,2)}(a_2)$. Clearly $H_{21}(z_{1,2})=0$ and so
\bea\label{alr.4}\sideset{}{_{i=1}^2}{\sum}\ol N_{(1,2)}(r,a_i;f)\leq N(r,0;H_{21})+S(r,f)\leq T(r, H_{21})+S(r,f)\leq T(r,\psi)+S(r,f).\eea 

Now from (\ref{bm.3}) and (\ref{alr.4}), we get
\bea\label{alr.5}\sideset{}{_{i=1}^2}{\sum}\ol N_{(1,2)}(r,a_i;f)\leq T(r,f)-N(r,\alpha;g)+S(r,f).\eea 

Then from (\ref{alr.5}), we have
\bea\label{alr.6}&&\sideset{}{_{i=1}^2}{\sum}\ol N_{(1,2)}(r,a_i;f)\leq T(r,f)-N(r,a_1;g)+S(r,f)\\ 
\text{and} 
\label{alr.7}&&\sideset{}{_{i=1}^2}{\sum}\ol N_{(1,2)}(r,a_i;f)\leq T(r,f)-N(r,a_2;g)+S(r,f).\eea  

Adding (\ref{alr.6}) and (\ref{alr.7}), we get 
\bea\label{alr.8}2\sideset{}{_{i=1}^2}{\sum}\ol N_{(1,2)}(r,a_i;f)\leq 2T(r,f)-N(r,a_1;g)-N(r,a_2;g)+S(r,f).\eea 

Now using (\ref{alr.3}), from (\ref{alr.8}), we get
\bea\label{alr.9} N(r,a_1;g)+N(r,a_2;g)\leq T(r,f)+S(r,f).\eea

Again from (\ref{alr.3}) and (\ref{alr.9}), we conclude that
\bea\label{alr.10} N(r,a_1;g)+N(r,a_2;g)\leq 2\sideset{}{_{i=1}^2}{\sum} \ol N_{(1,2)}(r,a_i;f)+S(r,f).\eea

Note that
$\sum_{1=1}^2\left(\ol N_{(1,1)}(r,a_i;f)+2\ol N_{(1,2)}(r,a_i;f)\right)\leq N(r,a_1;g )+N(r,a_2;g )+S(r,f)$
and so from (\ref{alr.10}) we get $\sum_{1=1}^2 \ol N_{(1,1)}(r,a_i;f)= S(r,f)$. Then
from (\ref{alr.1}), we get
\beas\label{alr.12} T(r,f)&=&\ol N_{(1,2)}(r,a_1;f)+\ol N_{(1,2)}(r,a_2;f)+S(r,f)\nonumber\\
&\leq& \ol N_{(2}(r,a_1;g)+\ol N_{(2}(r,a_2;g)+S(r,f)
\leq T(r,g )+S(r,f).\nonumber\eeas

Therefore from (\ref{bb.2}), we get $T(r,f)=T(r,g )+S(r,f)$, which is absurd by Lemma \ref{l4}.\par

{\bf Case 2.} Suppose $\phi\equiv 0$. Since $\delta(f)\not\equiv 0$, we get $f\equiv g$, i.e., $f(z)\equiv \mathscr{L}_k(f(z+c))$.
\end{proof}

\begin{proof}[{\bf Proof of Corollary \ref{c1}}] If $a_1$ and $a_2$ are constants, then $a_1'-a_2'\equiv 0$. 
Now if $\phi\equiv a_1'-a_2'$, then immediately $\phi\equiv 0$. 
Therefore if we follow the proof of Theorem \ref{t1}, then Sub-case 1.2 does not occur in the proof of Corollary \ref{c1}. Now
from the proof of Theorem \ref{t1}, we need only to consider the case when $\gamma_{2k}\equiv 0$. Clearly $\gamma_{2k}\equiv 0$ gives
\bea\label{t2.1} \sideset{}{_{i=1}^k}{\sum} b_i \psi_i=-b_0.\eea

Now from the recurrence formula (\ref{e3.4}) for $\psi_i$, we derive the expression
\bea\label{t2.2} \psi_i=\psi_1^i+R_{i-1}(\psi_1),\eea
where $R_{i-1}(\psi_1)$ is a differential polynomial. Then using (\ref{t2.2}), we get from (\ref{t2.1}) that 
\bea\label{t2.3} b_k \psi_1^k=\tilde R_{k-1}(\psi_1),\eea
where $\tilde R_{k-1}(\psi_1)$ is a differential polynomial in $\psi_1$.
Clearly $\phi$ is an entire function and so is $\psi_1$.

First suppose $\psi_1$ is transcendental. Then using Clunie lemma (see \cite{JC}) to (\ref{t2.3}), we get $m(r,\psi_1)=S(r,\psi_1)$. Since $N(r,\psi_1)=0$, it follows that $T(r,\psi_1)=S(r,\psi_1)$, which is impossible.\par

Next suppose $\psi_1$ is a polynomial. If $\psi_1$ is non-constant, then from (\ref{t2.3}), we get a contradiction. Hence $\psi_1$ is a constant.
\end{proof}

\begin{proof}[{\bf Proof of Corollary \ref{c2}}] From the proof of Theorem \ref{t1}, we need only to consider the case when $\phi\not\equiv a_1'-a_2'$. Now from (\ref{e3.3}), we get
\bea\label{t3.1} f^{(k)}(z+c)=\frac{\sideset{}{_{j=0}^{2k}}{\sum}\alpha_{k,j}f^j(z+c)+Q_k}{(f(z+c)-g(z+c))^{2k-1}},\eea
where
\beas Q_k=\sum\limits_{\substack{l<2k\\ l+j_1+j_2+\ldots+j_k\leq 2k}}\beta_{l, j_1, j_2,\ldots,j_i}f^l(g )^{j_1}(g ')^{j_2}\ldots (g ^{(2k-1)})^{j_k}.\eeas

Here $\alpha_{k,j}, \beta_{l, j_1, j_2,\ldots,j_l}\in S(f)$ and $\psi_p:=\alpha_{p,2p}$ satisfies the recurrence formula defined by (\ref{e3.4}).

Now we divide the following two cases.\par

{\bf Case 1.} Let $\psi_k=\alpha_{k,2k}\not\equiv 0$ for $p=1,2,\ldots,k-1$. Now following the same procedure as done in Sub-case 1.1 in the proof of Theorem \ref{t1}, we get a contradiction.\par
 
{\bf Case 2.} Let $\psi_k=\alpha_{k,2k}\equiv 0$. Now from (\ref{e3.4}), we get $\psi_{k-1}'+\psi_1 \psi_{k-1}\equiv 0$.
On integration, we have $\psi_{k-1}(z)=c_0e^{\xi(z)}$, where $\xi(z)=-\int_{0}^z \psi_1(z)\;\text{d}z$ is a non-constant meromorphic function and $c_0\in\mathbb{C}\backslash \{0\}$. We know that if $\xi(z)$ has a pole at the point $z_0$, then $z_0$ is an essential singularity of $e^{\xi(z)}$. Since $\psi_{k-1}$ is a non-constant meromorphic function, we deduce that $\xi$ is a non-constant entire function and so $\psi_1$ is also non-constant entire function. Now if $\psi_1$ is a polynomial, then $\psi_{i+1}$ is also a polynomial for $i=1,2,\ldots, k-1$. Therefore we immediately get a contradiction. Hence $\psi_1$ is a transcendental entire function. 
Now from (\ref{t2.2}), we get $b_k\psi_1^k=\tilde R_{k-1}(\psi_1)$.
By Clunie lemma (see \cite{JC}), we get $m(r,\psi_1)=S(r,\psi_1)$. Since $N(r,\psi_1)=0$, we have $T(r,\psi_1)=S(r,\psi_1)$, which is impossible.
\end{proof}

\begin{rem} We can easily conclude that Theorem \ref{t1} and Corollaries \ref{c1}-\ref{c3} are still valid for non-constant meromorphic function satisfying $N(r,f) = S(r, f)$.
\end{rem}

\medskip
%\section{{\bf Statements and declarations}}
%\vspace{1.3mm}
{\bf Statements and declarations:}

\smallskip
\noindent \textbf {Conflict of interest:} The authors declare that there are no conflicts of interest regarding the publication of this paper.

\smallskip
\noindent{\bf Funding:} There is no funding received from any organizations for this research work.

\smallskip
\noindent \textbf {Data availability statement:}  Data sharing is not applicable to this article as no database were generated or analyzed during the current study.

\end{document}